\documentclass[letterpaper,10 pt,conference]{ieeeconf}  % Comment this line out if you need a4paper
\IEEEoverridecommandlockouts                              % This command is only needed if 
\usepackage{comment}

\usepackage{subfigure}

\usepackage{xcolor}
\usepackage{xparse}
\usepackage{mathrsfs}
\usepackage{graphics} % for pdf,  bitmapped graphics files
\usepackage{amsmath} % assumes amsmath package installed
\usepackage{amssymb}  % assumes amsmath package installed
\usepackage{amsthm}
\usepackage{enumitem} 
\usepackage{hyperref}%
\hypersetup{colorlinks=true,  linkcolor=blue,  breaklinks=true,  urlcolor=blue,  citecolor=blue}

\usepackage{soul}
\usepackage{url}
\usepackage[prependcaption, colorinlistoftodos]{todonotes}

\DeclareMathOperator*{\rank}{rank}

\newcommand\oprocendsymbol{\hbox{$\triangle$}}
\newcommand\oprocend{\relax\ifmmode\else\unskip\hfill\fi\oprocendsymbol}

\DeclareSymbolFont{bbold}{U}{bbold}{m}{n}
\DeclareSymbolFontAlphabet{\mathbbold}{bbold}
\newtheorem{theorem}{Theorem}
\newtheorem{remark}[theorem]{Remark}
\newtheorem{example}[theorem]{Example}

\newtheorem{definition}[theorem]{Definition}
\newtheorem{proposition}[theorem]{Proposition}

\newtheorem{problem}{Problem}

\newcommand {\be}{\begin{equation}}
\newcommand {\ee}{\end{equation}}

\newcommand{\im}{{\rm Im}}

\definecolor{personal}{rgb}{0,0.5,0.5}

\title{\LARGE\bf  Behaviors, trajectories and data: \\ A novel perspective on the design of unknown-input observers}
\author{Giorgia Disar\`o and Maria Elena Valcher
\thanks{G. Disar\`o and  M.E. Valcher are with the Dipartimento di Ingegneria dell'Informazione,
 Universit\`a di Padova,
    via Gradenigo 6B, 35131 Padova, Italy, e-mail:  \texttt{giorgia.disaro@phd.unipd.it, meme@dei.unipd.it}}
   }
  \date{}

\begin{document}
\maketitle

\begin{abstract}
The purpose of this paper is to propose a novel perspective, based on Willems' ``behavior theory", on the design of  an unknown-input observer for a given linear time-invariant discrete-time state-space model, with unknown disturbances affecting both the state and the output equations. The problem is first addressed assuming that the original system model is known, and later assuming that the model is unknown but historical data satisfying a certain assumption are available. 
In both cases, fundamental concepts in behavior theory, as the projection of a behavior, the inclusion of a behavior in another one, 
and the use of kernel and image representations, provide quite powerful tools to determine necessary and sufficient conditions for the 
existence of an unknown-input observer (UIO), as well as algorithms to design one of them, if it exists.
\end{abstract}

\medskip

\section{Introduction} \label{intro}

Control engineering problems 
have been traditionally approached from a model-based perspective, assuming that some approximated model   describing
the system dynamics is available, possibly after a preliminary identification phase.
% and then designing controllers, observers or solving other
%control problems based on the existing model.
 Among all possible representations, linear state-space models surely have been 
the most commonly used, due to the fact that, despite their simplicity,   they can  effectively 
retain  in a small neighbour of the  working conditions  the  essential system properties one needs to design appropriate control actions.
 However, when dealing with complex systems integrating components of different nature, 
  it is often difficult to a priori determine  the parameters
   characterizing the system model, or even the class of  models that better describe the system.
  The uncertainty on the model, or even on the model class, one 
 has to adopt in order to capture the system dynamics,
  has led over the years to the development of methodologies that formulate the problem solution based only on what one is able to observe, namely finite trajectories generated by the system or, equivalently, some raw data.

  The idea of defining a system as the set of the trajectories compatible with its governing laws is at the basis of the ``behavioral approach" to systems theory, 
  first proposed by J.C. Willems in the eighties  \cite{Willems86,Willems86-II} and subsequently explored in a series of successful papers (see \cite{wil91,Willems}, to cite a few). In Willems' perspective, a system {\em is} the set of its trajectories, i.e., the {\em behavior}.  
So, 
  trajectories come first and from trajectories one can deduce the most suitable representation.
  % In addition, one of the distinctive features of the behavioral approach is to avoid imposing a priori a cause-effect relationship between sets of variables, but to deduce it from the properties of the system representation derived from the system trajectories.
 
  More recently, there has been a surge of interest in the so-called 
data-driven techniques, as opposed to the commonly adopted  model-based ones. 
Generally speaking, data-driven methods exploit collected data to design controllers, observers or to solve any kind of control problem. They split into two categories: indirect and direct methods. The former ones use the available data to derive a model representation, performing a preliminary system identification procedure, and then they apply standard model-based approaches to solve the problem. The latter ones, instead, avoid the model identification step and leverage only the collected data to obtain the problem solution. The direct approach has recently gained more attention since it has the possibility of outperforming the two-steps approach. 
%Indeed, by avoiding the identification step, it does not introduce a bias error due to a possibly wrong selection of the model representation. 
Moreover, it has  the advantage of being simpler and computationally more efficient. 

The connection between behavior theory and data-driven methods is actually very strong, and not unexpectedly a large number of results obtained 
with data-driven methods \cite{DD_Behaviors,MarkRap, Maupong,Ferrari-Trecate, Henk24} rely on the so-called Willems' Fundamental Lemma (and the concept of persistence of excitation) \cite{WillemsPE}, and on its extensions \cite{WillemsExt,WillemsExt2}. 
Indeed, the absence of an underlying model representing the system is what makes behaviorial systems theory \cite{wil91,Willems86,Willems86-II} particularly suitable for use in the formulation of data-driven problems with direct approach. 
A system can be identified with the set of its trajectories, and, under certain conditions, it suffices a single  finite-time trajectory to capture the behavior. 
This is the key idea behind the development of data-driven methods: 
to design an experiment in such a way that the collected data   are ``sufficiently rich" to either identify  the system or to 
solve some specific control problem for it.

As previously mentioned, data-driven formulations based on behavioral systems theory have already been proposed to solve classical control engineering problems.
%see \cite{DD_Behaviors,MarkRap,Maupong,Henk24}, exploiting in particular Willems' Fundamental Lemma \cite{WillemsPE} 
However, to the best of the authors' knowledge, the behavioral paradigm has not been considered yet to formulate and solve in a data-driven context the problem of designing an unknown-input observer (UIO) for a linear time-invariant (LTI) discrete-time system. 

%More specifically, in this paper we adopt a direct data-driven approach, but assuming to know in advance that the data-generating system can be described as a state space model. Not only, we assume that also the UIO admits a state space description. This implies that we are implicitly assuming that both the original system and the observer are causal or, equivalently, have a (strictly) proper transfer function. 
%In general, in the behavioral context,   causality   is not required or imposed a priori.
%, mainly because
%, as already recalled, 
%a cause-effect relationship is   typically derived and not assumed.
 %Indeed, 
 Various types of estimation problems have been investigated in the behavioral approach, 
 however the causality of  the obtained (unknown-input) observer is not always guaranteed \cite{MB-MEV-JCW,vw2,BehaviorObservers}. In particular, in \cite{MB-MEV-JCW} the general problem of designing an observer of some relevant variables from some measured variables in the presence of  unmeasured (and/or irrelevant) variables is investigated for linear left shift-invariant behaviors whose trajectories have support on $\mathbb Z_+$. 
 %The considered problem  amounts to the design of a UIO if the disturbances are the only unmeasured - and irrelevant - variables appearing in the system. 
 Necessary and sufficient conditions for the existence of a dead-beat observer are provided as well as a complete parametrization of all dead-beat observers. Then, equivalent conditions for the existence of a {\em causal} dead-beat observer are derived, which  can be particularized to address the state estimation problem for  state-space models. However, in
 general, they cannot be used to design a classic UIO described by a state-space model.
 
 In  this paper we address the UIO design problem by resorting again  to the behavioral approach, but following a   different perspective from the one adopted in \cite{MB-MEV-JCW,vw2,BehaviorObservers}, since we constrain the resulting UIO to be amenable of a classic state-space description.
 In the first part of the paper, by adopting for both the  state-space model representing the original system and the one representing the UIO a ``kernel description", and leveraging   the concepts of acceptor \cite{BehaviorObservers} 
 and of projection of a behavior on a subset of its variables, we derive novel necessary and sufficient conditions for the problem solvability, and relate them to the standard ones derived using state-space methods \cite{Darouach2,UIO-MEV}. 
Then, in the second part of the paper,   we exploit the obtained results to tackle the problem in a data-driven scenario. 

The   data-driven design of a UIO has  been investigated in \cite{TAC-UIO,ECC2024,Ferrari-Trecate}, and necessary and sufficient conditions for the problem solvability have been derived. However, all  previous works   considered a simpler system description compared to the  one adopted here,   as the unknown input was  not affecting  the output measurements. Moreover,
by resorting to kernel and image representations that are typical of the behavioral approach, 
we have been able to derive a different problem solution compared to those given in \cite{TAC-UIO,ECC2024,Ferrari-Trecate} that 
is more amenable to generalizations and also provides an easy algebraic algorithm
to design the matrices of a possible UIO, if it exists.
 
% in all these papers the data-driven formulation of the solution is derived following the standard algebraic formulation. Instead, in this paper we characterize the data-driven solution exploiting the concepts of behavioral systems theory, which is very well suited to deal with data. More specifically, we assume to observe during a preliminary experiment finite-time state/input/output trajectories generated by the original system, corresponding to (known and unknown) inputs capable of exciting the system dynamics, so that these trajectories contain all the relevant information about the system, and can thus be used as an alternative description of it. In particular, we identify the projection of the system on the relevant and measured variables with the image of the matrix containing the observed finite-time trajectories. Then, we derive an equivalent kernel representation of this matrix that, leveraging the concept of acceptor \cite{BehaviorObservers}, can be exploited to provide necessary and sufficient conditions for the existence of a UIO for the system, as well as a practical way to determine its matrices. 

%It is worthwhile highlighting that the approach presented here leads to a novel characterization of the solution to the problem of designing a UIO based only on some collected data with respect to the ones already obtained in the literature, for instance in . {\color{red} Devo ancora scrivere bene la parte sul confronto}****

The paper is organized as follows. In Section \ref{problem} we formalize the problem we are going to address. In Section \ref{behavior_not} we provide some basic notions about behaviors.
We believe that such concepts are sufficient to make even an interested reader not familiar with behavior theory in a position to understand the subsequent analysis. 
In Section \ref{behavior_app} we apply behavioral systems theory to solve the problem of designing a UIO described by a   state-space model. The same problem is solved also in Section \ref{datadriven} by making use only of some collected data. An illustrative example is provided in Section \ref{example}. Finally, Section \ref{Concl} concludes the paper. 
\medskip

{\bf Notation.} 
Given two integers, $h \le k$, the symbol $[h,k]$ denotes the finite set $\{h, h+1, \dots, k\}$.
Given a matrix $M\in {\mathbb R}^{p \times m}$, 
 $[M]_{ij}$ denotes the $(i,j)$-th entry of $M$. By $e_i$ we denote the $i$-th canonical vector in ${\mathbb R}^n$ (with all entries equal to zero except for the $i$-th one which is unitary), $n$ being clear from the context.
%  and {\color{red} serve??{\color{personal} Non mi pare, avevo soltanto messo un po' di notazione utilizzata di solito} }  $M^\dag\in {\mathbb R}^{m \times p}$ its {\em Moore-Penrose inverse} \cite{BenIsraelGreville}. Note that if $M$ is of full row rank, then $M^\dag = M^\top( MM^\top)^{-1}$. A symmetric result holds if $M$ is of full column rank. 
The null and column spaces of $M$ are denoted by $\ker{(M)}$ and $\im (M)$, respectively. \\
Given a vector sequence $v =\{v(t)\}_{t\in {\mathbb Z}_+}$, taking values in $\mathbb{R}^n$, we use the notation $\{v(t)\}_{t=0}^{N}$, $N\in\mathbb Z _+$, to indicate the (finite) sequence of vectors $v(0),\dots,v(N)$.  \\
We denote by ${\mathbb R}[z]$ the ring of polynomials in the variable $z$ with real coefficients. 
Let $M(z)\in\mathbb R[z] ^{p\times q}$ be a polynomial matrix.  
The {\em rank} of $M(z)$ is the rank of the (complex-valued) matrix  $M(\lambda)$ %{\it almost everywhere}, namely
 for every $\lambda \in {\mathbb C}$, except (possibly) for a finite number of values of $\lambda$.
We say that $M(z)$ is {\em left prime} \cite{Kailath}
%(resp., {\em left prime}) if $\rank{(M(\lambda))} = q$ (resp.,
if  $\rank{(M(\lambda))} = p$ for every $\lambda \in \mathbb C$.   %Similarly, one can define the concept of right prime matrix. 

A square polynomial matrix $U(z)\in \mathbb R[z]^{p\times p}$ is {\em unimodular} if it is invertible and its inverse is a polynomial matrix, too. 
%A unimodular matrix is both left prime and right prime.
If $M(z)\in\mathbb R[z] ^{p\times q}$ has rank $r$, a polynomial matrix $H(z)$ is a {\em left annihilator} of $M(z)$ (see \cite{MB-MEV-JCW, Rocha, BehaviorsZ+}) if $H(z)M(z)=0$. Moreover, a left annihilator $H_m(z)$ is a {\em minimal left annihilator (MLA)} of $M(z)$ if it is of full row rank and for any other left annihilator $H(z)$ of $M(z)$ we have $H(z) = P(z)H_m(z)$, for some polynomial matrix $P(z)$. It can be proved that an MLA of $M(z)$  is a $(p-r)\times p$ left prime matrix and is uniquely determined modulo a unimodular left factor. 

Given a polynomial   vector $u(z)\in {\mathbb R}[z]^p$ by its {\em degree} we mean the highest of the degrees of its (polynomial) entries, i.e.,
$\max_j \deg (u_j(z)).$
% {\em Right annihilators} and {\em minimal right annihilators (MRAs)} are defined in an analogous way and enjoy analogous properties.
A polynomial matrix $M(z)\in\mathbb R[z]^{p\times q}$ of rank $p$ with row degrees\footnote{The $i$-th {\em row degree} $k_i$ is the 
degree of $e_i^\top M(z)$,   the $i$-th row of $M(z)$.} $k_1,k_2,\dots,k_p$ is said to be {\em row-reduced} \cite{Kailath} if there exists at least one minor of order $p$ whose degree coincides with $\sum_{i=1}^{p}{k_i}$. When so,   $M(z)$   has the {\em predictable degree property} (see \cite[Theorem 6.3-13]{Kailath}), by this meaning that    for every $u(z)\in\mathbb R[z]^p$ it holds that 
$$
{\rm deg}(u(z)^\top M(z)) = \max_{j:u_j(z)\ne 0}{\{{\rm deg}(u_j(z))+k_j\}}.
$$
 \smallskip
 
\section{Problem formulation}\label{problem}

Consider a discrete-time LTI system, $\Sigma$, described by:
\begin{subequations}\label{system}
\begin{eqnarray}\label{system_1}
x(t+1) &=& Ax(t)+Bu(t)+Ed(t) \\ 
y(t) &=& Cx(t)+Du(t)+Fd(t), \label{system_2}
\end{eqnarray}
\end{subequations}
where $t\in\mathbb{Z}_+$, $x(t)\in \mathbb{R}^n$ is the state, $u(t)\in\mathbb{R}^m$ is the (known) control input, $y(t)\in\mathbb{R}^p$ is the output and $d(t)\in\mathbb{R}^r$ is the unknown input of the system, e.g., a disturbance. Without loss of generality, % (w.l.o.g), 
we assume that the matrix $\begin{bmatrix} E^\top & F^\top\end{bmatrix}^\top \in\mathbb{R}^{(n+p)\times r}$ is of full column rank
 $r$. Indeed, if  this is not the case, 
% and 
%$\rank{(\begin{bmatrix} E^\top & F^\top\end{bmatrix}^\top)}=\bar r<r$, we can always rewrite it as $\begin{bmatrix} E^\top & F^\top\end{bmatrix}^\top =\begin{bmatrix} \bar E^\top & \bar F^\top\end{bmatrix}^\top V $, where $\begin{bmatrix} \bar E^\top & \bar F^\top\end{bmatrix}^\top\in\mathbb{R}^{(n+p)\times \bar r}$ is a full column rank matrix and $V\in\mathbb{R}^{\bar r \times r}$ is a full row rank matrix and define a new unknown input $\bar{d} (t) \triangleq Vd(t)$. 
one can simply redefine the disturbance signal to enforce the full column rank assumption. \\
As a first step, we introduce the concepts of 
acceptor (first adopted in the  behavioral setting \cite{BehaviorObservers}, and later introduced also in a data-driven scenario   \cite{TAC-UIO,ECC2024}) and  of unknown input observer.
\smallskip

\begin{definition}\label{UIO}\cite{TAC-UIO}
An LTI system of the form 
\begin{subequations}\label{UIO.eq}
\begin{eqnarray}\label{UIO_eq1}
z(t+1) &=& A_{UIO}z(t)+B_{UIO}^u u(t)+B_{UIO}^y y(t) \\
\hat{x}(t) &=& z(t) +  D_{UIO}^u u(t) +  D_{UIO}^y y(t), \label{UIO_eq2}
\end{eqnarray}
\end{subequations}
where   $z(t)\in\mathbb R^n$ is the state and $\hat x(t)\in\mathbb R^n$  is the output, is said to be:
\begin{itemize}
\item  An {\em acceptor} for system $\Sigma$ in \eqref{system} if for every trajectory $(\{x(t)\}_{t\in\mathbb Z_+}, \{u(t)\}_{t\in\mathbb Z_+},\{y(t)\}_{t\in\mathbb Z_+},\{d(t)\}_{t\in\mathbb Z_+})$~that is a solution of \eqref{system}, there exists $\{z(t)\}_{t\in\mathbb Z_+}$ such that $(\{\hat x(t)\}_{t\in\mathbb Z_+}=\{x(t)\}_{t\in\mathbb Z_+},\{u(t)\}_{t\in\mathbb Z_+},\{y(t)\}_{t\in\mathbb Z_+},$ $\{z(t)\}_{t\in\mathbb Z_+})$ is a solution of \eqref{UIO.eq}.
\item An  {\em  unknown-input observer (UIO)} for the system in \eqref{system} if it is an acceptor and the state estimation error $e(t) \triangleq x(t) - \hat{x} (t)$ tends to 0 as $t\to +\infty$, for every choice of $x(0)$, $z(0)$,  of the input signal  $u(t), t\in {\mathbb Z}_+$, and     of the unknown input $d(t), t\in {\mathbb Z}_+$.
\end{itemize}
\end{definition}

%Assuming to have full knowledge of the system description, 
The UIO design problem can be formalised as follows:

\begin{problem} \label{probl}
Given system $\Sigma$ described as in \eqref{system}, design (if possible) a UIO  for $\Sigma$, described by equations \eqref{UIO.eq}.
\end{problem}

%In the next sections we will provide the problem solution  by exploiting 
%first the behavioral approach and then a data-driven approach.
%{\color{red} and we will show that, under suitable assumptions on the collected data, the solvability conditions are equivalent. LO FACCIAMO?}
Before addressing the behavioral approach to the problem solution,
we  recall some basics about   behavior theory. The interested reader is referred, e.g.,  to \cite{Willemsbook,Willems} for  details.
\medskip

\section{Some basic notions about behaviors} \label{behavior_not}

Let $(\mathbb R^{\texttt w})^{\mathbb Z_+}$ be the set of trajectories defined on $\mathbb Z_+$ and taking values in $\mathbb R^{\texttt w}$. The {\em left (backward) shift operator} on $(\mathbb R^{\texttt w})^{\mathbb Z_+}$ is defined as 
\begin{eqnarray*}
\sigma\!&:&\! (\mathbb R^{\texttt w})^{\mathbb Z_+}\to (\mathbb R^{\texttt w})^{\mathbb Z_+}\\
\!&:&\! (w(0),w(1),w(2),\dots)\mapsto (w(1),w(2),w(3),\dots),
\end{eqnarray*}
for any $w= \{w(t)\}_{t\in \mathbb Z_+}\in (\mathbb R^{\texttt w})^{\mathbb Z_+}$.
%In this paper we  adopt the following definition of behavior. 
\begin{definition} \label{behavior} \cite{BehaviorsZ+, Willems86}
A {\em behavior} $\mathcal B\subseteq(\mathbb R^{\texttt w})^{\mathbb Z_+}$ is the linear and left shift invariant set of solutions $w = \{w(t)\}_{t\in\mathbb Z_+}$ of a system of difference equations 
$$
R_0  w(t) + R_1 w(t+1) +\dots + R_L w(t+L) = 0, \quad \forall\ t\in {\mathbb Z}_+,
$$
or, equivalently, in compact form, the set of solutions of
$
R(\sigma) w = 0,
$
where $R(z) \triangleq \sum_{i=0}^{L}{R_iz^i}$ belongs to $\mathbb R[z]^{p\times {\texttt w}}$, $p,  L\in \mathbb Z_+, p >0$. 
We adopt the standard shorthand notation $$\mathcal B \triangleq \ker{(R(\sigma))},$$
and refer to   the vector $w(t)$ as the vector of the system   variables (at time $t$).  
\end{definition}

Given a behavior $\mathcal B\subseteq(\mathbb R^{\texttt w})^{\mathbb Z_+}$ and a positive time   $T\in {\mathbb Z}_+$, 
   the symbol $\mathcal B\big|_{[0,T-1]}$   denotes the restriction of ${\mathcal B}$ to the time window $[0,T-1]$, namely
 the set of finite sequences $\{w_t\}_{t=0}^{T-1}$ such that there exists $w\in {\mathcal B}$ satisfying  $w(t)=w_t$ for every $t\in [0,T-1]$.

A behavior $\mathcal B\subseteq(\mathbb R^{\texttt w})^{\mathbb Z_+}$  is {\em autonomous} \cite{BehaviorsZ+,Wood-Zerz}
 if there exists $l \in {\mathbb Z}_+$ such that every trajectory $w$ in $\mathcal B$ is uniquely determined by its first $l$ samples 
 $\{w(t)\}_{t=0}^{l-1}$.
 %$w(0), w(1), \dots, w(l-1)$.
This is the case if and only if ${\mathcal B}$ is the kernel of a full column rank polynomial matrix $R(z)$. When so, it entails no loss of generality assuming that $R(z)$ is nonsingular square. An autonomous behavior 
$\mathcal B = \ker{(R(\sigma))}$ is {\em asymptotically stable} if all its trajectories converge to $0$ as $t$ goes to $+\infty$, and this is the case if and only if the greatest common divisor of the maximal order minors of $R(z)$ is a Schur polynomial.

\begin{definition} \label{projection} Given a behavior $\mathcal B\subseteq(\mathbb R^{\texttt w})^{\mathbb Z_+}$, 
assume that the vector of the  system variables $w(t)$ can be split into two blocks, say $w(t) = (w_1(t),w_2(t))\in \mathbb R^{\texttt w}$, where $w_1(t) \in \mathbb R^{\texttt w_1}$, $w_2(t) \in \mathbb R^{\texttt w_2}$, and ${\tt w}_1+{\tt w}_2={\tt w}$. 
The {\em projection of} $\mathcal B$ {\em on the   variables} $w_1$ is defined as 
$$\mathcal P_{w_1} \mathcal B \triangleq \{w_1 : \exists w_2 \in (\mathbb R^{\texttt w_2})^{\mathbb Z_+} {\rm \ such \ that \ } (w_1,w_2)\in \mathcal B\}.$$
\end{definition} 

If ${\mathcal B}$ is described as 
$[R_1(\sigma) \ R_2(\sigma) ] \begin{bmatrix}
w_1 \cr  w_2\end{bmatrix} =0,$
%$R_1(\sigma) w_1 + R_2(\sigma) w_2 =0,$
for  suitable polynomial matrices $R_1(z)$ and $R_2(z)$ (with the same number of rows), then \cite{MB-MEV-JCW}
$\mathcal P_{w_1} \mathcal B = {\rm ker} (M_2(\sigma) R_1(\sigma)),$ where $M_2(z)$ is an MLA of $R_2(z)$.
\medskip

\section{Behavioral approach to the design of a UIO} \label{behavior_app} 

Consider the state-space model $\Sigma$, and set $w = (x,u,y,d)$. The behavior associated to the system, $\mathcal B_{\Sigma}$, can be described as 
$
R_{\Sigma}(\sigma)w = 0,
$
or, equivalently, as
\be \label{system_B}
\mathcal B_{\Sigma} \triangleq \ker{(R_{\Sigma}(\sigma))},
\ee
where 
\be \label{R_sigma}
R_{\Sigma}(z) \triangleq 
\left[\begin{array}{cccc}
z I_n - A & -B & 0 & -E \\
%\hline
-C & -D & I_p & -F 
\end{array}\right].
\ee
Since, by hypothesis,  $\begin{bmatrix}
E^\top & 
F^\top
\end{bmatrix}^\top$ is a real matrix of full column rank
%, i.e., $\rank{(\begin{bmatrix}
%E^\top & 
%F^\top
%\end{bmatrix}^\top)} = r$, 
$r$,
an MLA of $\begin{bmatrix}
E^\top & 
F^\top
\end{bmatrix}^\top$  is simply a full row rank constant matrix,
say 
$\begin{bmatrix} M_E & M_F
\end{bmatrix}
\in \mathbb R^{(n+p-r)\times (n+p)}$, such that 
${\rm ker}\begin{bmatrix} M_E & M_F
\end{bmatrix} = {\rm Im} \begin{bmatrix}
E^\top & 
F^\top
\end{bmatrix}^\top,$
and hence

{\small\begin{align}
 & {\mathcal P}_{(x,u,y)}{\mathcal B}_{\Sigma} = 
 \ker \left(\begin{bmatrix} M_E & M_F
\end{bmatrix} \begin{bmatrix}
\sigma I_n -   A  & -  B &0 \cr
- C & - D & I_p
\end{bmatrix}\right)
 \\
&= \ker \big(\begin{bmatrix}
\sigma M_E - M_E A -M_F C & -M_E B 
- M_F D & M_F
\end{bmatrix}\big). \nonumber
\end{align}}
%\end{array}\right. }\\
% \left.
%&& \begin{array}{c}
%& -I_r \\
%\hline
%& 0
%\end{array}
%\right]\Big)}

Upon defining the  describing vector of system $\hat \Sigma$ in \eqref{UIO.eq} as $\hat w = (\hat x, u, y, z)$,   the associated behavior is: 
%$$
%R_{\hat \Sigma}(\sigma)\hat w(t) = 0,
%$$
%or, equivalently, 
\be \label{UIO_B}
\mathcal B_{\hat \Sigma} \triangleq \ker{(R_{\hat \Sigma}(\sigma))},
\ee
where 
\be \label{R_sigma_hat}
R_{\hat \Sigma}(z) \triangleq \left[\begin{array}{cccc}
 0 & -B_{UIO}^u & -B_{UIO}^y & z I_n - A_{UIO} \\
I_n & -D_{UIO}^u & -D_{UIO}^y & -I_n 
\end{array}\right].
\ee
By premultiplying \eqref{R_sigma_hat} by  $ \left[\begin{array}{c|c} I_n & z I_n - A_{UIO} \end{array}\right] \in \mathbb R[z]^{n\times 2n}$, which is 
(left prime and) an MLA of $\begin{bmatrix}
z I_n -A_{UIO}^\top &
-I_n
\end{bmatrix}^\top$, we obtain the projection of $\mathcal B_{\hat \Sigma}$ on the 
 variables $(\hat x,u,y)$:

{\small
\begin{eqnarray}
\!\!\!\!\ &\!\!\!\!\!\!\mathcal P_{(\hat x,u,y)}\mathcal B_{\hat \Sigma} = \ker{\big(\big[\begin{array}{c|c}
\sigma I_n - A_{UIO}  & -B_{UIO}^u 
\end{array}} \label{proj_stim}\\
\!\!\!\!\ &\!\!\!\!\!\!\begin{array}{c|c}
- (\sigma I_n - A_{UIO}) D_{UIO}^u & -B_{UIO}^y - (\sigma I_n - A_{UIO})D_{UIO}^y
\end{array}\big]\big). \nonumber
\end{eqnarray}
}

The system in \eqref{UIO.eq} is  a UIO of $\Sigma$ if and only if the following two conditions are satisfied: 
\begin{enumerate}
\item 
\be \label{P_incl}
\mathcal P_{(x,u,y)}\mathcal B_{\Sigma} \  \subseteq \ \mathcal P_{(\hat x,u,y)}\mathcal B_{\hat \Sigma}.
\ee
\item If $(x,u,y) \in \mathcal P_{(x,u,y)}\mathcal B_{\Sigma}$ and $(\hat x,u,y)\in \mathcal P_{(\hat x,u,y)}\mathcal B_{\hat \Sigma}$, then $e = x-\hat x$ must be a trajectory of an autonomous asymptotically stable behavior, say $\mathcal B_e$.
% = \ker{\left(\sigma I - A_{UIO}\right)}\equiv \ker{\left(T(\sigma) (\sigma M_E - M_E A -M_F C)\right)}$.
\end{enumerate}
Note that    condition \eqref{P_incl} amounts to saying  that if $(x,u,y)$ is a trajectory of  $\mathcal P_{(x,u,y)}\mathcal B_{\Sigma}$, there always exists $z$ 
 such that $(x,u,y,z)$ belongs to
$\mathcal B_{\hat \Sigma}$
  %such that $\hat x = x$ 
  and thus $(x, u, y)$ is   a trajectory of $\mathcal P_{(\hat x,u,y)}\mathcal B_{\hat \Sigma}$.
This means that $\hat \Sigma$ is an acceptor for the system in \eqref{system}.
We first address condition 1).  

\begin{proposition} \label{prop_acc}
$\hat \Sigma$ is an acceptor for $\Sigma$
(i.e, condition \eqref{P_incl} holds) 
if and only if the matrices $A_{UIO}, B_{UIO}^u, B_{UIO}^y, D_{UIO}^u$ and $D_{UIO}^y$ satisfy the following conditions:

{\small\begin{eqnarray}
\!\!\!\!\!&&\!\!\!\!\!\!\!\!\begin{bmatrix}
- D_{UIO}^y & A_{UIO}D_{UIO}^y-B_{UIO}^y
\end{bmatrix}
\begin{bmatrix}
CE & F \\
F & 0 
\end{bmatrix}
= 
\begin{bmatrix}
-E & 0
\end{bmatrix}\  \  \label{acc1}\\
\!\!\!\!\!&&\!\!\!\!\!\!\!\! A_{UIO} = A + \begin{bmatrix}
-D_{UIO}^y & A_{UIO}D_{UIO}^y-B_{UIO}^y
\end{bmatrix}
\begin{bmatrix}
CA \\
C
\end{bmatrix} \label{acc2}\\
\!\!\!\!\!&&\!\!\!\!\!\!\!\!\begin{bmatrix}
B_{UIO}^u &  D_{UIO}^u
\end{bmatrix}= 
\begin{bmatrix}
I - D_{UIO}^yC &  - B_{UIO}^y\cr
0 & - D_{UIO}^y
\end{bmatrix} \begin{bmatrix}
B \\
D
\end{bmatrix}. \  \label{acc3}
\end{eqnarray}}
\end{proposition}

\begin{proof} 
The inclusion in \eqref{P_incl}  holds if and only if
%\footnote{Let $A$ and $B$ be two matrices with the same number of columns. Then 
(see the Completion Lemma in \cite{Willems86}) there exists $T(z)\in \mathbb R[z]^{n\times (n+p-r)}$ such that
%\begin{eqnarray*}
%&&{\rm ker} (A) \subseteq {\rm ker} (B) \ \Leftrightarrow \ 
%({\rm ker} (A))^\perp \supseteq ({\rm ker} (B))^\perp 
%\ \Leftrightarrow \\
%&& \Leftrightarrow\ 
%{\rm Im} (A^\top) \supseteq {\rm Im} (B^\top)
%\ \Leftrightarrow \ 
%B^\top = A^\top X, \exists \ X \\
%&& \Leftrightarrow  \ B = TA, \exists \ T.
%\end{eqnarray*}
%The reasoning, in particular, holds if $A$, $B$ and $T$ are polynomials matrices.
%} 
\begin{eqnarray} \label{proj_incl}
\!\!\!\!\!\!&&\!\!\!\!\!\big[\begin{array}{c|c|c} 
z I_n - A_{UIO}  & -B_{UIO}^u  - (z I_n - A_{UIO}) D_{UIO}^u & -B_{UIO}^y \end{array} \nonumber\\
\!\!\!\!\!\!&&\!\!\!\!\!\begin{array}{c}- (z I_n - A_{UIO})D_{UIO}^y 
\end{array}\big] = T(z) \cdot  \begin{bmatrix} M_E & M_F
\end{bmatrix}\\
\!\!\!\!\!\!&&\!\!\!\!\! \cdot \ \left[\begin{array}{c|c|c}
z I_n - A & -B & 0  \\
\hline
-C & -D & I_p 
\end{array}\right], \nonumber
\end{eqnarray}
which is equivalent, block by block, to
  the following conditions:
\begin{subequations} \label{T_eq}
\begin{eqnarray}
\!\!\!\!\!\!&&\!\!\!\!\!\!\!\!\!\!\!\!\!\!\!\!\!\!zI_n- A_{UIO} \!=\! T(z)[M_E(zI_n-A)-M_F C] \label{T1}\\
\!\!\!\!\!\!&&\!\!\!\!\!\!\!\!\!\!\!\!\!\!\!\!\!\! B_{UIO}^u + (zI_n-A_{UIO})D_{UIO}^u \!=\! T(z)(M_E B+ M_F D) \label{T2} \\
\!\!\!\!\!\!&&\!\!\!\!\!\!\!\!\!\!\!\!\!\!\!\!\!\! B_{UIO}^y+(zI_n-A_{UIO})D_{UIO}^y \!=\! - T(z)M_F. \label{T3}
\end{eqnarray}
\end{subequations}
In order to solve the previous equations, we want 
to preliminarily  determine the degree of $T(z)$.
The right hand side of \eqref{proj_incl} can   be written as 
$
S(z) G(z),
$
where $S(z) \triangleq T(z) \left[\begin{array}{c|c} M_E & M_F \end{array}\right]$, while the polynomial matrix 
$$G(z)\triangleq\left[\begin{array}{c|c|c}
z I_n - A & -B & 0  \\
\hline
-C & -D & I_p 
\end{array}\right]$$ 
is row-reduced, with row degrees $k_j \in \{0,1\}$,  $j\in[1,n+p]$, 
 and hence has the predictable degree property. This ensures that, for every $i\in[1,n]$,
 
{\small \begin{eqnarray} \label{pred_deg}
&&\!\!\!\!\!\!\!\!\!\!\!1 = {\rm deg}\Big(e_i^\top\big[\begin{array}{c|c|} 
z I_n - A_{UIO}  & -B_{UIO}^u  - (z I_n - A_{UIO}) D_{UIO}^u  \end{array} \nonumber\\
&&\!\!\!\!\!\!\!\!\!\!\!\!\!\!\!\!\!\!\!\!\!\begin{array}{c}\ \ \ \ -B_{UIO}^y- (z I_n - A_{UIO})D_{UIO}^y 
\end{array}\big] \Big) = \\
&&\!\!\!\!\!\!\!\!\!\!\! =\!\!\! \max_{j\in[1,n+p]: \atop
[S(z)]_{ij} \ne 0}{\Big\{{\rm deg}\left( [S(z)]_{ij}\right) +k_j \Big\}}. \nonumber
%&&\!\!\!\!\!\!\!\!\!\!\!\!\!\!\!\!\!\!\!\!\! =\!\!\! \max_{j\in[1,n+p]}{\Big\{{\rm deg}\Big[e_i^\top\big[T(z)\left[\begin{array}{c|c} M_E & M_F \end{array}\right]\big]e_j+k_j\Big]\Big\}},
\end{eqnarray}}

\noindent It follows that ${\rm deg}\left([S(z)]_{ij}\right) \le 1$, for every $ i\in[1,n]$ and $j\in[1,n+p]$,
%can be either $0$ or $1$. 
 namely $S(z)= T(z)\left[\begin{array}{c|c} M_E & M_F \end{array}\right]$ is a polynomial matrix of degree (at most) $1$, and since 
$\left[\begin{array}{c|c} M_E & M_F \end{array}\right]$ is of full row rank, this implies that $T(z)= T_0+T_1z$ with $T_0, T_1 \in\mathbb R^{n\times (n+p-r)}$. \\
%{\color{blue} Moreover, it is also possible to prove that this is the only way in which we can express the solution to \eqref{proj_incl}. Indeed, suppose that $\tilde T(z) = \sum_{k=0}^{l}{\tilde T_kz^k}$ is another solution of higher degree. By performing some simple calculations, we obtain $\tilde T_k \left[\begin{array}{c|c} M_E & M_F \end{array}\right] = 0, \ \forall k=2,\dots,l$. Since $\left[\begin{array}{c|c} M_E & M_F \end{array}\right]$ is of full row rank, we get $\tilde T_k = 0$ for every  $k=2,\dots,l$, and thus $\tilde T(z) = \tilde T_0+\tilde T_1z$. } \\
Consequently, the conditions in \eqref{T_eq} are equivalent to the following system of equations: 
\begin{subequations} \label{conditions}
\begin{eqnarray}
\!\!\!\!\!\!\!\!\!\!\!\!\!\!\!&&\!\!\!\!\!\!\!\!\!\! T_0M_E = I-D_{UIO}^yC \label{A}\\
\!\!\!\!\!\!\!\!\!\!\!\!\!\!\!&&\!\!\!\!\!\!\!\!\!\! T_0M_F = A_{UIO}D_{UIO}^y - B_{UIO}^y \label{B} \\
\!\!\!\!\!\!\!\!\!\!\!\!\!\!\!&&\!\!\!\!\!\!\!\!\!\! T_1M_E = 0 \label{C} \\
\!\!\!\!\!\!\!\!\!\!\!\!\!\!\!&&\!\!\!\!\!\!\!\!\!\! T_1M_F = -D_{UIO}^y \label{D} \\
\!\!\!\!\!\!\!\!\!\!\!\!\!\!\!&&\!\!\!\!\!\!\!\!\!\! A_{UIO} = T_0M_EA + T_0M_FC \nonumber \\
\!\!\!\!\!\!\!\!\!\!\!\!\!\!\!&&\!\!\!\!\!\!\!\!\!\! \ \ \ \ \ \ \ \ = (I-D_{UIO}^yC) A + (A_{UIO}D_{UIO}^y-B_{UIO}^y)C \label{E} 
\end{eqnarray}
\begin{eqnarray}
\!\!\!\!\!\!\!\!\!\!\!\!\!\!\!&&\!\!\!\!\!\!\!\!\!\! B_{UIO}^u = T_0M_EB+A_{UIO}D_{UIO}^u+T_0M_FD \nonumber \\
\!\!\!\!\!\!\!\!\!\!\!\!\!\!\!&&\!\!\!\!\!\!\!\!\!\! \ \ \ \ \ \ \ \ =(I-D_{UIO}^yC)B + A_{UIO}D_{UIO}^u \nonumber \\
\!\!\!\!\!\!\!\!\!\!\!\!\!\!\!&&\!\!\!\!\!\!\!\!\!\! \ \ \ \ \ \ \ \ + (A_{UIO}D_{UIO}^y- B_{UIO}^y) D \label{F}  \\
\!\!\!\!\!\!\!\!\!\!\!\!\!\!\!&&\!\!\!\!\!\!\!\!\!\!D_{UIO}^u  = T_1M_EB+T_1M_FD =-D_{UIO}^yD. \label{G}
\end{eqnarray}
%in which the last three equations have not been assigned a number since they are exactly the same as the ones in \eqref{SS_cond4}$\div$\eqref{SS_cond6}. 
Moreover, since $ \left[\begin{array}{c|c} M_E & M_F \end{array}\right]$ is an MLA of $\begin{bmatrix}
E^\top & 
F^\top
\end{bmatrix}^\top$, it follows that 
\begin{eqnarray}
\!\!\!\!\!\!\!\!\!\!\!\!\!\!\!\!\!\!\!\!\!\!\!0 \!\!\!\!\!\!&=&\!\!\!\!\!\! T_0M_E E +T_0M_F F \nonumber \\
\!\!\!\!\!\!\!\!\!\!\!\!\!\!\!\!&\stackrel{{\small \eqref{A}-\eqref{B}}}{=}&\!\!\!\!\!\!(I-D_{UIO}^yC)E+(A_{UIO}D_{UIO}^y-B_{UIO}^y)F \label{H} \\ 
\!\!\!\!\!\!\!\!\!\!\!\!\!\!\!\!\!\!\!\!\!\!\!0 \!\!\!\!\!\!&=&\!\!\!\!\!\! T_1M_E E +T_1M_F F \nonumber \\
\!\!\!\!\!\!\!\!\!\!\!\!\!\!\!\!&\stackrel{{\small \eqref{C}-\eqref{D}}}{=}&\!\!\!\!\!\! -D_{UIO}^y F. \label{I}
\end{eqnarray}
\end{subequations}
By re-ordering equations \eqref{E}$\div$\eqref{I},  we obtain equations  \eqref{acc1}$\div$\eqref{acc3}.
 \end{proof}
\medskip
%This shows that the conditions obtained in the behavioral context are identical to the ones obtained in the model-based approach. 

Proposition \ref{prop_acc}
 shows that   system $\hat \Sigma$ described by equations \eqref{UIO.eq} is an acceptor for   system $\Sigma$  in \eqref{system}
if and only if conditions \eqref{E}$\div$\eqref{I}  hold. 
It remains, now,  to impose condition 2), namely that $\hat \Sigma$ asymptotically tracks the state of $\Sigma$.  

\begin{proposition} \label{prop_uio}
An acceptor $\hat \Sigma$ for $\Sigma$ is a UIO (namely satisfies condition 2)) if and only if $A_{UIO}$ is Schur.
\end{proposition}
\begin{proof}
 Condition 2) amounts to saying that
for every $(x,u,y)\in \mathcal P_{(x,u,y)}\mathcal B_{\Sigma}$ and every $(\hat x, u,y)\in  \mathcal P_{(\hat x,u,y)}\mathcal B_{\hat \Sigma}$, the error signal  $e = x-\hat x$ converges to zero asymptotically.
As $\hat \Sigma$ is an acceptor,  $(x,u,y)\in     \mathcal P_{(\hat x,u,y)}\mathcal B_{\hat \Sigma}$, in turn, and hence, by linearity, it follows that 
$(e,0,0)= (x,u,y)-(\hat x, u,y)\in   \mathcal P_{(\hat x,u,y)}\mathcal B_{\hat \Sigma}$,
which means that
$e \in \ker{\left(
\sigma I - A_{UIO}
\right)}. $
%=\ker
%\left[T(\sigma)(\sigma M_E - M_E A -M_F C)\right]$.
Consequently, %we need to impose that 
$\mathcal B_e  \triangleq  \ker{\left(
\sigma I - A_{UIO}
\right)}$
 must be  autonomous and asymptotically stable, i.e.,    $A_{UIO}$ must be  Schur stable.
%This, in particular, requires that $z M_E - M_E A -M_F C$ has a Schur polynomial as greatest common divisor of its minors of order $n$. 
\end{proof}
\smallskip

Putting together Propositions \ref{prop_acc} and \ref{prop_uio}, we obtain the following result.
%If we now rewrite the conditions in \eqref{E},\eqref{H} and \eqref{I} as 
%\begin{eqnarray*}
%\!\!\!\!\!&&\!\!\!\!\!\!\!\!\begin{bmatrix}
%- D_{UIO}^y & A_{UIO}D_{UIO}^y-B_{UIO}^y
%\end{bmatrix}
%\begin{bmatrix}
%CE & F \\
%F & 0 
%\end{bmatrix}
%= 
%\begin{bmatrix}
%-E & 0
%\end{bmatrix} \\
%\!\!\!\!\!&&\!\!\!\!\!\!\!\! A_{UIO} = A + \begin{bmatrix}
%-D_{UIO}^y & A_{UIO}D_{UIO}^y-B_{UIO}^y
%\end{bmatrix}
%\begin{bmatrix}
%CA \\
%C
%\end{bmatrix}
%\end{eqnarray*}
%and we define 
%, summarising the analysis of the current and previous sections.

\begin{theorem}\label{IS_UIO}
The following facts are equivalent. 
\begin{itemize}
\item[(i)] $\hat \Sigma$ is  a UIO %of the form \eqref{UIO.eq} 
for   $\Sigma$.
%There exists a UIO of the form \eqref{UIO.eq} for system $\Sigma$ 
\item[(ii)] The following two conditions are satisfied: 
\begin{itemize}
\item[a)] $\mathcal P_{(x,u,y)}\mathcal B_{\Sigma} \  \subseteq \ \mathcal P_{(\hat x,u,y)}\mathcal B_{\hat \Sigma}$, and 
\item[b)] $\mathcal B_e$ is an autonomous asymptotically stable behavior. 
\end{itemize}
\item[(iii)] The 
matrices 
 $A_{UIO}, B_{UIO}^u, B_{UIO}^y, D_{UIO}^u$ and $D_{UIO}^y$ satisfy \eqref{acc1}$\div$\eqref{acc3}, and $A_{UIO}$ is Schur stable.
\end{itemize}
\end{theorem} 

%The previous analysis has led to the same set of conditions derived in the literature by resorting to a state-space model approach.
The previous result allows us to identify under what conditions a system described as in 
\eqref{UIO.eq} 
is a UIO for $\Sigma$. Now we want to determine under what  conditions on $\Sigma$ a UIO $\hat \Sigma$ exists.
%this can be seen as the problem of determining under what conditions there exists $T(z) = T_0 + T_1 z$ such that
%equation \eqref{proj_incl} holds for some real matrices $A_{UIO}, B_{UIO}^u, B_{UIO}^y, D_{UIO}^u$ and $D_{UIO}^y$ 
%with $A_{UIO}$ Schur stable.
We have the following result.

\begin{theorem}\label{CNS_UIO}
The following facts are equivalent. 
\begin{itemize}
\item[(i)] There exists    a UIO $\hat \Sigma$ of the form \eqref{UIO.eq} 
for   $\Sigma$.
%There exists a UIO of the form \eqref{UIO.eq} for system $\Sigma$ 
\item[(ii)] There exists $T(z) = T_0 + T_1 z\in {\mathbb R}[z]^{n \times (n+p-r)}$ such that
equation \eqref{proj_incl} holds for some real matrices $A_{UIO}, B_{UIO}^u, B_{UIO}^y, D_{UIO}^u$ and $D_{UIO}^y$,
with $A_{UIO}$ Schur stable.
\item[(iii)] There exists $T(z) = T_0 + T_1 z\in {\mathbb R}[z]^{n \times (n+p-r)}$ such that
\be
T(z) \begin{bmatrix} M_E & M_F\end{bmatrix} \begin{bmatrix}
zI_n-A  \\
- C  
\end{bmatrix}  =  zI_n - A_{UIO},
\label{intermediate}
\ee
for some Schur stable matrix $A_{UIO}$.
\item[(iv)] There exists $S(z) = S_0 + S_1 z\in {\mathbb R}[z]^{n \times (n+p)}$ such that 
\be
S(z) \begin{bmatrix}
zI_n-A & -E \\
- C & - F 
\end{bmatrix}  = \begin{bmatrix} zI_n - A_{UIO} &0\end{bmatrix},
\ee
for some Schur stable matrix $A_{UIO}$.
\item[(v)] There exist
matrices 
 $A_{UIO}, B_{UIO}^u, B_{UIO}^y, D_{UIO}^u$ and $D_{UIO}^y$, with  $A_{UIO}$   Schur stable, satisfying \eqref{acc1}$\div$\eqref{acc3}.
\end{itemize}
\end{theorem}

\begin{proof}
{\em (i)} $\Leftrightarrow$ {\em (ii)} It follows immediately from the proof of Proposition \ref{prop_acc} and from Proposition \ref{prop_uio}.
\\
 {\em (ii)} $\Leftrightarrow$ {\em (iii)}  It is straightforward.
 \\
 {\em (iii)} $\Leftrightarrow$ {\em (iv)} 
 If {\em (iii)} holds for some $T(z) = T_0 + T_1z$, by setting
$S(z) = T(z) \begin{bmatrix} M_E & M_F\end{bmatrix}$ we immediately deduce {\em (iv)}. Conversely, if 
{\em (iv)} holds, then 
$$S(z) \begin{bmatrix}
E \\
 F 
\end{bmatrix}  =  0,$$
which immediately implies that $S_i= T_i   \begin{bmatrix} M_E & M_F\end{bmatrix}$, for some $T_i, i\in \{0,1\}$.
This immediately proves   {\em (iii)}.
\\
{\em (i)} $\Leftrightarrow$ {\em (v)} It follows immediately from Theorem \ref{IS_UIO}.
 
\end{proof}

%{\color{red} mi sembra un po' piu' carino cosi' perche' resta con flavor behavior ma non sono riuscita a provare che una delle precedenti condizioni e' equivalente a quelle del remark qua sotto. Al momento questa e' una soluzione paraculo che potremmo vendere accompagnata dalla salsa "due to page limits the proof is omitted.." ;-)}
%{\color{personal} Mi sembra che questa paraculata sia un'ottima soluzione ;-)} 
%namely under what conditions 
%matrices 
% $A_{UIO}, B_{UIO}^u, B_{UIO}^y, D_{UIO}^u$ and $D_{UIO}^y$ as in {\em (iii)} of Theorem \ref{IS_UIO} exist,
% once we set $L \triangleq \begin{bmatrix} -D_{UIO}^y & A_{UIO}D_{UIO}^y-B_{UIO}^y \end{bmatrix}$ this amounts 
% to determine necessary and sufficient conditions for the existence 
%of 
% $L\in\mathbb R^{n\times2p}$ such that 
%\begin{itemize}
%\item[a)] $L\begin{bmatrix}
%CE & F \\
%F & 0 
%\end{bmatrix}
%= 
%\begin{bmatrix}
%-E & 0
%\end{bmatrix}$, and
%\item[b)] the matrix $A+L \begin{bmatrix}
%CA \\
%C
%\end{bmatrix}$ is Schur stable.
%\end{itemize}
%Indeed, once we have found a matrix $L$ as in such that a) and b) hold, the matrices describing the UIO in \eqref{UIO.eq} are uniquely determined.  Indeed, see \eqref{acc1}-\eqref{acc2}, from the first $p$ columns of the matrix $L$ we obtain the matrix $D_{UIO}^y$. Then, from the remaining columns of $L$ and the already determined expressions of $A_{UIO}$ and $D_{UIO}^y$, we obtain $B_{UIO}^y$. From equations \eqref{F}-\eqref{G}, the matrices $B_{UIO}^u$ and $D_{UIO}^u$ are uniquely determined in turn.  

\begin{remark} \label{darouach_comp}
It is possible to prove (but details are omitted, due to page constraints) that the conditions given in Theorem \ref{CNS_UIO}  are equivalent to those obtained in \cite{Darouach2} using only algebraic manipulations on the state-space models, namely 
\begin{itemize}
\smallskip
\item[(a)] $\rank{\left( \begin{bmatrix}
zI_n-A & -E \\
C & F 
\end{bmatrix}\right)} = n + \rank{\left( \begin{bmatrix}
E \\
F
\end{bmatrix} \right)}= n+ r$,  $\forall z \in \mathbb C, |z|\ge 1$;
\smallskip
\item[(b)] $\rank{\left( \begin{bmatrix}
CE & F \\
F & 0
\end{bmatrix}\right)} = \rank{(F)} + \rank{\left( \begin{bmatrix}
E \\
F
\end{bmatrix} \right)}= \rank{(F)}+r$.
%{\color{red} Questa sarebbe anche equivalente alla condizione di FCR della matrice conduttrice per il primo pezzo della matrice row-reduced?} 
\end{itemize}
A system $\Sigma$ satisfying (a) and (b) is strong* detectable (see \cite[Theorem 2]{Darouach2}).  
\end{remark}

\begin{remark}
It is worth remarking that 
the analysis carried on in this section could not be deduced from the   purely behavioral one carried on in 
\cite{MB-MEV-JCW}, where there was no constraint that the derived solution could be realised through a state-space model.
%In fact, in that paper no causality constraint was imposed. 
\end{remark}
\smallskip
%{\color{red} Allo stesso risultato su $\mathcal B_e$ si poteva giungere considerando il behavior dei due sistemi interconnessi definendo la variabile e e ricavando la proiezione su e }

\section{Data-driven approach to the design of a UIO} \label{datadriven}
In this section we tackle   Problem 1 in a data-driven framework. As already done in \cite{TAC-UIO,Ferrari-Trecate}, we assume to have performed an \emph{offline} experiment collecting some state-input-output trajectories in the time interval $[0,T-1]$ with $T\in\mathbb{Z}_+$ and we define the following sequences of vectors
 $x_d =  \{x_d(t)\}_{t=0}^{T-1}$, $u_d =  \{u_d(t)\}_{t=0}^{T-1}$ and $y_d = \{y_d(t)\}_{t=0}^{T-1}$,
that we call \emph{historical} data. 
The assumption that the  state is available for measurements during the preliminary phase of data collection has proved to be necessary for the design of a UIO, since input/output data alone would not  allow one to identify the dimension and the basis of the state-space.  Moreover, 
an {\em ad hoc} preliminary test can be conceived, possibly involving additional sensors, to make the state measurement  possible, while 
it may not be realistic to use such sensors in standard working conditions. 
This justifies the fact that we assume to be able to measure the state offline, but not online.
More on this issue can be found in \cite{TAC-UIO,Ferrari-Trecate}.
In addition, although  we have no access to the unknown input $d(t)$, for the subsequent analysis it will be useful to introduce a symbol for the sequence of historical unknown input data, i.e., 
  $d_d = \{d_d(t)\}_{t=0}^{T-1}$. 
We rearrange the above data into the following matrices: 
\begin{eqnarray*}
X_p &\triangleq& \begin{bmatrix}
x_d(0) & \dots & x_d(T-2)
\end{bmatrix} \in {\mathbb R}^{n \times (T-1)}, \\
X_f &\triangleq &\begin{bmatrix}
x_d(1) & \dots & x_d(T-1)
\end{bmatrix} \in {\mathbb R}^{n \times (T-1)}, \\
U_p &\triangleq &\begin{bmatrix}
u_d(0) & \dots & u_d(T-2)
\end{bmatrix} \in {\mathbb R}^{m \times (T-1)}, \\
U_f &\triangleq &\begin{bmatrix}
u_d(1) & \dots & u_d(T-1)
\end{bmatrix} \in {\mathbb R}^{m \times (T-1)}, \\
Y_p &\triangleq &\begin{bmatrix}
y_d(0) & \dots & y_d(T-2)
\end{bmatrix}\in {\mathbb R}^{p \times (T-1)}, \\
Y_f &\triangleq &\begin{bmatrix}
y_d(1) & \dots & y_d(T-1)
\end{bmatrix} \in {\mathbb R}^{p \times (T-1)}, \\
D_p &\triangleq& \begin{bmatrix}
d_d(0) & \dots & d_d(T-2)
\end{bmatrix} \in {\mathbb R}^{r \times (T-1)}, \\
D_f &\triangleq& \begin{bmatrix}
d_d(1) & \dots & d_d(T-1)
\end{bmatrix} \in {\mathbb R}^{r \times (T-1)},
\end{eqnarray*}
where the subscripts $p$ and $f$ stand for past and future, respectively. 
We also introduce the following notation:
$$
{\footnotesize\Phi_d \triangleq \begin{bmatrix}
X_p^\top &\!\!
X_f ^\top &\!\!
U_p ^\top &\!\!
U_f ^\top &\!\!
Y_p^\top &\!\!
Y_f ^\top
\end{bmatrix}^\top\!\! \in {\mathbb R}^{2(n+m+p)\times (T-1)}.}
$$
Whenever we deal with data-driven techniques, we know that it is important for the data to be representative of the system that has generated them. More formally, we give the following definition,  analogous to those given in  \cite{TAC-UIO,Ferrari-Trecate}. 

\begin{definition} \label{compatibility}
A trajectory $(x,u,y)\in ({\mathbb R}^{n+m+p})^{{\mathbb Z}_+}$ is said to be {\em compatible with the historical data} $(\{x_d(t)\}_{t=0}^{T-1},\{u_d(t)\}_{t=0}^{T-1},\{y_d(t)\}_{t=0}^{T-1})$ if 
\be \label{image}
\begin{bmatrix}
x(t) \\
x(t+1) \\
u(t) \\
u(t+1) \\
y(t) \\
y(t+1) 
\end{bmatrix} \in {\rm Im}(\Phi_d), \quad \forall t \in \mathbb Z_+.
\ee
\end{definition}

% It is possible to show, by adapting the analysis carried out in \cite{TAC-UIO,Ferrari-Trecate} to the more complex system in \eqref{system}, that the compatibility of the trajectories with the historical data is guaranteed under the following:   
In the sequel we assume that the data satisfy the following:

\smallskip
\noindent {\bf Assumption:}
The matrix $\begin{bmatrix}
X_p^\top &
U_p^\top &
U_f^\top &
D_p^\top & 
D_f^\top 
\end{bmatrix}^\top$ is of full row rank.
\smallskip 

  It worth highlighting that, even if the previous assumption cannot be checked in practice using the available data (that do not include the unknown input data), it is satisfied, for instance, under the hypothesis that system $\Sigma$ is controllable and the  (known and unknown) inputs  are persistently exciting of order $n+2$ \cite{WillemsPE}. The known input can always be chosen to this purpose. Instead, for what concerns the unknown input, the requirement is still not verifiable. However, if $T$ is sufficiently large, it is reasonable to assume that 
the   disturbance has this property, since it typically  changes randomly (see also Remark 6 in \cite{TAC-UIO} and Remark 4 in \cite{Ferrari-Trecate}).
 
Under the previous Assumption, by adapting the analysis carried out in \cite{TAC-UIO,Ferrari-Trecate} to the more complex system in \eqref{system}, it is possible to prove 
that the (state/input/output)  trajectories of $\Sigma$ coincide with those
 compatible  with the historical data. More formally, upon defining   
$$
\mathbb T_c \triangleq  \{(x,u,y)\in ({\mathbb R}^{n+m+p})^{{\mathbb Z}_+}% \in ({\mathbb R}^{n+m+p})^{{\mathbb Z}_+}
\ {\rm satisfying}\ \eqref{image} \},
% \begin{bmatrix}
%x(t) \\
%x(t+1) \\
%u(t) \\
%u(t+1) \\
%y(t) \\
%y(t+1) 
%\end{bmatrix} \in {\rm Im} (\Phi_d), \forall t\ge 0\right\},
$$
we can state the following result. 

\begin{proposition} \label{compat}
Under the Assumption on the data, the state/input/output trajectories compatible with the historical data are all and only those generated by the system, i.e.,  
$$
\mathbb T_c \equiv \mathcal P_{(x,u,y)}\mathcal B_{\Sigma}. 
$$ 
\end{proposition}

%\begin{proof}
%We first show that $\mathcal P_{(x,u,y)}\mathcal B_{\Sigma} \supseteq \mathbb T_c$. If $(x,u,y)\in \mathbb T_c$, then there exists $g_t\in\mathbb R^{(T-1)}$ such that 
%$$
%\begin{bmatrix}
%x(t) \\
%x(t+1) \\
%u(t) \\
%u(t+1) \\
%y(t) \\
%y(t+1) 
%\end{bmatrix} = 
%\begin{bmatrix}
%X_p \\
%X_f \\
%U_p \\
%U_f \\
%Y_p \\
%Y_f 
%\end{bmatrix}g_t
%$$
%***
%\end{proof}
%
\medskip

%Under the previous assumption on the data, it is possible to adapt the analysis carried out in \cite{TAC-UIO} to the more complex system in \eqref{system}. 
Theorem \ref{ker_incl}, below, provides necessary and sufficient conditions, based on    data, 
for the existence of a UIO. % for $\Sigma$.
\smallskip

%More specifically, we can claim that the following lemma holds. The proof is omitted since it is analogous to the one provided in \cite{TAC-UIO}.  
\begin{theorem}\label{ker_incl}
 There exists a UIO of the form \eqref{UIO.eq} for the system $\Sigma$ in \eqref{system} if and only if 
 for every choice of    real matrices $V_p, V_f, W_p,  W_f, R_p$ and $R_f$ of suitable sizes such that
\be
\label{ker_im}
\!\!\! 
\ker \left(  \begin{bmatrix} V_p & V_f & W_p & W_f & R_p & R_f \end{bmatrix} \right) = {\rm Im}(\Phi_d), \!
\ee
 there exists a real matrix $\Omega$ such that
\be \Omega \begin{bmatrix} V_p & V_f\end{bmatrix} = \begin{bmatrix}- A^* & I_n\end{bmatrix},
\label{sol_cond}
\ee
where $A^*\in {\mathbb R}^{n\times n}$ is a Schur matrix.
\end{theorem}

\begin{proof}
{\em Only if.} Suppose that there exists a  UIO, $\hat \Sigma$,  for system $\Sigma$ in \eqref{system}, described as  in \eqref{UIO.eq}. This means that 
 $\mathcal P_{(x,u,y)}\mathcal B_{\Sigma} \  \subseteq \ \mathcal P_{(\hat x,u,y)}\mathcal B_{\hat \Sigma}$ and that $A_{UIO}$ is Schur stable. Therefore,  
 if $w = (x,u,y)$ is a trajectory belonging to $\mathcal P_{(x,u,y)}\mathcal B_{\Sigma}$, it also belongs to $ \mathcal P_{(\hat x,u,y)}\mathcal B_{\hat \Sigma}$. 
 Since the historical data have been generated by system $\Sigma$, this in particular implies that $(\{x_d(t)\}_{t=0}^{T-1},(\{u_d(t)\}_{t=0}^{T-1},(\{y_d(t)\}_{t=0}^{T-1})\in\mathcal P_{(\hat x,u,y)}\mathcal B_{\hat \Sigma}\big|_{[0,T-1]} $ and thus for every $t\in[0,T-2]$ we have 
 \begin{eqnarray*}
\!\!\!\!\!\!\!\!\!&& \big[\begin{array}{c|c|}
\sigma I_n - A_{UIO}  & -B_{UIO}^u 
- (\sigma I_n - A_{UIO}) D_{UIO}^u 
\end{array} \\
\!\!\!\!\!\!\!\!\!&&\begin{array}{c}
 -B_{UIO}^y - (\sigma I_n - A_{UIO})D_{UIO}^y
\end{array}\big] \begin{bmatrix}
x_d(t) \\
u_d(t) \\
y_d(t) 
\end{bmatrix} = 0,
\end{eqnarray*}
and therefore
\begin{eqnarray}
\!\!\!\!\!\!\!\!\!&& \big[\begin{array}{c|c|c|c|}
- A_{UIO}  &  I_n &   -B_{UIO}^u+A_{UIO}D_{UIO}^u & -D_{UIO}^u \nonumber
\end{array}\ \  \ \ \\
\!\!\!\!\!\!\!\!\!&&\begin{array}{c|c}
 - B_{UIO}^y+A_{UIO}D_{UIO}^y & -D_{UIO}^y 
\end{array}\big] \Phi_d = 0. \label{ortog}
\end{eqnarray}
This means that ${\rm Im}(\Phi_d)$ is included in the kernel of the matrix appearing on the left hand-side in \eqref{ortog}.
If we select any  matrix $\begin{bmatrix} V_p & V_f & W_p & W_f & R_p & R_f \end{bmatrix}$ such that \eqref{ker_im} holds,
then there exists  a matrix $\Omega$ of suitable size such that
\begin{eqnarray*}
\!\!\!\!\!\!\!\!\!&&\big[\begin{array}{c|c|c|c|}
- A_{UIO}  &  I_n &   -B_{UIO}^u+A_{UIO}D_{UIO}^u & -D_{UIO}^u
\end{array} \\
&&\begin{array}{c|c}
  - B_{UIO}^y+A_{UIO}D_{UIO}^y & -D_{UIO}^y 
\end{array}\big] =\\
&&\ \ \ \  = \Omega \begin{bmatrix} V_p & V_f & W_p & W_f & R_p & R_f \end{bmatrix}.
\end{eqnarray*}
This implies 
 that \eqref{sol_cond} holds for $A^*= A_{UIO}$.
\medskip

{\em If.} Let $V_p, V_f, W_p,  W_f, R_p$ and $R_f$ be real matrices  of suitable sizes such that 
 \eqref{ker_im} holds and suppose that there exists a real matrix $\Omega$ such that
\eqref{sol_cond} holds for some Schur stable matrix $A^*$.
Set 
\begin{eqnarray}
&&\!\!\!\!\!\!  \begin{bmatrix} - A^* & I_n &  - S_3 & -S_4 & -S_5 & -S_6 \end{bmatrix}  \nonumber \\
&& \ \ \ \ \ \ \  \triangleq \Omega \begin{bmatrix} V_p & V_f & W_p & W_f & R_p & R_f \end{bmatrix}.
\label{multiplo} \end{eqnarray}
Clearly, 
\begin{eqnarray*}
&&\ker \left(\begin{bmatrix} - A^*& I_n &  - S_3 & -S_4 & -S_5 & -S_6 \end{bmatrix} \right) \supseteq \\
&& \ker \left(\begin{bmatrix} V_p & V_f & W_p & W_f & R_p & R_f \end{bmatrix}\right) = {\rm Im} (\Phi_d).
\end{eqnarray*}
 We want to show that  there exists a system of the form \eqref{UIO.eq} that is a UIO for system $\Sigma$. 
%From \eqref{multiplo}, we deduce  that 
%\be \label{Xf5T}
%X_f = \left[\begin{array}{c|c|c|c|c}
%S_1 & S_2 & S_3 & S_4 & S_5
%\end{array}\right]
%\begin{bmatrix}
%X_p\\
%U_p \\
%U_f \\
%Y_p \\
%Y_f 
%\end{bmatrix}, 
%\ee
%or equivalently, 
%\be \label{withXf}
% \left[\begin{array}{c|c|c|c|c|c}
%-S_1 & -S_2 & -S_3 & -S_4 & -S_5 & I_n
%\end{array}\right]
%\begin{bmatrix}
%U_p \\
%U_f \\
%Y_p \\
%Y_f \\
%X_p \\
%X_f
%\end{bmatrix} = 0, 
%\ee
%with $S_1,S_2\in\mathbb R^{n\times m}$, $S_3,S_4\in\mathbb R^{n\times p}$ and $S_5\in\mathbb R^{n\times n}$. \\
Under the Assumption we made on the data, we have that  every trajectory $(x,u,y) \in {\mathcal P}_{(x,u,y)}{\mathcal B}_{\Sigma}$ satisfies \eqref{image} 
for every $t\ge 0$,
%$$\begin{bmatrix}
%x(t) \\
%x(t+1) \\
%u(t) \\
%u(t+1) \\
%y(t) \\
%y(t+1) 
%\end{bmatrix} \in {\rm Im} (\Phi_d),$$
and hence
$$ \label{conXf}
\begin{bmatrix} - A^*& I_n &  - S_3 & -S_4 & -S_5 & -S_6 \end{bmatrix} 
\begin{bmatrix}
x(t) \\
x(t+1) \\
u(t) \\
u(t+1) \\
y(t) \\
y(t+1) 
\end{bmatrix} = 0.$$
 This implies that the trajectory $(x,u,y) \in {\mathcal P}_{(x,u,y)}{\mathcal B}_{\Sigma}$ belongs to the kernel of the following matrix 
\be \label{zT}
\left[\begin{array}{c|c|c}
 z I_n - A^* & -z S_4 - S_3 & -z S_6 -S_5 
\end{array}\right].
\ee
Upon setting 
\begin{equation} \label{Si}
\begin{array}{l}
 A_{UIO} \triangleq  A^*,   \
D_{UIO}^u \triangleq S_4, \
B_{UIO}^u \triangleq  S_3+A^*S_4\\
 D_{UIO}^y \triangleq S_6, \
B_{UIO}^y  \triangleq S_5+A^*S_6,
\end{array} 
%\begin{eqnarray}
% A_{UIO} &\triangleq&  A^*, \label{S1} \\
%D_{UIO}^u &\triangleq& S_4 \\
%B_{UIO}^u &\triangleq&  S_3+A^*S_4\\
% D_{UIO}^y &\triangleq& S_6\\
%B_{UIO}^y  &\triangleq& S_5+A^*S_6,
%\end{eqnarray} 
\end{equation}
the polynomial matrix in \eqref{zT} coincides with $\big[\begin{array}{c|c|c}
z I_n - A_{UIO}  & -B_{UIO}^u 
- (z I_n - A_{UIO}) D_{UIO}^u  & -B_{UIO}^y \end{array}
\\ \begin{array}{c}  - (z I_n - A_{UIO})D_{UIO}^y
\end{array}\big].$
Therefore if we select a system $\hat \Sigma$, described as in
\eqref{UIO.eq} for the previous choice of the matrices $A_{UIO}, B_{UIO}^u, B_{UIO}^y, D_{UIO}^u$ and $D_{UIO}^y$,
we know that \eqref{proj_stim} holds, and we have just proved that 
$$ {\mathcal P}_{(x,u,y)}{\mathcal B}_{\Sigma} \subseteq \mathcal P_{(\hat x,u,y)}\mathcal B_{\hat \Sigma}.$$
Consequently, $\hat \Sigma$ is  an acceptor of the form \eqref{UIO.eq} for   $\Sigma$.  Finally, since $A_{UIO}=A^*$ is Schur stable,   $\hat \Sigma$ is a UIO.
 \end{proof}

%{\color{red} Definire la concatenazione di behaviors }

%So far, we have only addressed necessary and sufficient conditions for the existence of an acceptor for the original system or, equivalently, we have only showed under what conditions a system $\hat \Sigma$, described as in
%\eqref{UIO.eq} exists corresponding to which 
% the error behavior  $\mathcal B_e$ is autonomous. Now in order for the system in \eqref{UIO.eq} to be a UIO, we have also to guarantee that the state estimation error asymptotically converges to zero, namely that $B_e$ is also asymptotically stable. 
%
%The general solution of equation \eqref{Xf5T} can be expressed as 
%\be \label{gen_sol}
%\!\!\left[\begin{array}{c|c|c|c|c}
%T_1 & T_2 & T_3 & T_4 & S_1
%\end{array}\right] \!\!=\! X_f \!\!
%\begin{bmatrix}
%U_p \\
%U_f \\
%Y_p \\
%Y_f \\
%X_p
%\end{bmatrix}^\dag \!\!+ Z \left(\!\! I \!\!-\begin{bmatrix}
%U_p \\
%U_f \\
%Y_p \\
%Y_f \\
%X_p
%\end{bmatrix}\!\!\!\begin{bmatrix}
%U_p \\
%U_f \\
%Y_p \\
%Y_f \\
%X_p
%\end{bmatrix}^{\!\dag}\right)
%\ee
%for some matrix $Z\in\mathbb R^{(T-1)\times (2m+2p+n)}$. 
%\smallskip
%\\ 
%%{\color{red} La matrice $\begin{bmatrix}
%%U_p \\
%%U_f \\
%%Y_p \\
%%Y_f \\
%%X_p
%%\end{bmatrix}$ sarebbe il prodotto di due matrici, di cui la seconda FRR. Mi sembra che l'altro giorno avessimo concluso che in tal caso la pseudoinversa del prodotto sia il prodotto trasposto delle pseudoinverse. Riprovando a dimostrarlo però non mi torna la quarta proprietà, ovvero $(A^\dag A)^\top = A^\dag A$. }
\medskip

\begin{remark} \label{practical}
It is worthwhile noticing that the proof of the previous Theorem \ref{ker_incl} 
provides also an elementary way to practically determine the matrices of a UIO from the historical data, if it exists.
Indeed, it is sufficient to first determine any matrix
$\begin{bmatrix} V_p & V_f & W_p & W_f & R_p & R_f \end{bmatrix}$
such that \eqref{ker_im} holds. 
Then search for a real matrix  $\Omega$, if it exists,  such that  
\eqref{sol_cond} holds.
To this end, we first verify if $V_f$ is of full column rank. If not, the problem is not solvable. If $\rank(V_f)= n$, we can parametrize the set of solutions of $I_n = \Omega V_f$ as
$\Omega = \bar \Omega + L \Delta_f,$
where $\bar \Omega$ is a specific left inverse of $V_f$, $\Delta_f$ is a full row rank matrix such that ${\rm ker}(\Delta_f)= {\rm Im} (V_f)$, and $L$ is a free parameter.
Once we set
$\bar A \triangleq  \bar \Omega V_p,$ and $\bar C \triangleq \Delta_f V_p,$
we   reduce ourselves to the problem of finding $L$ such that 
$\bar A + L \bar C$ is Schur stable.\\
If such $L$ exists, then, by making use of \eqref{multiplo} and \eqref{Si}, one deduces the matrices of a possible UIO.
\end{remark}
 \medskip

\section{Example} \label{example}

As a proof of concept,  we now  illustrate the results obtained in the previous sections by means of an elementary numerical example.

\begin{example} \label{ex2}
Consider a system $\Sigma$  of order $n=3$ described as in \eqref{system} for the following choice of matrices: 
\begin{align*}
A  &= \begin{bmatrix}   
     1  &   1   & -1 \\
     2   &  1   &  1 \\
     1   &  0   & -1
\end{bmatrix}, \quad
B = \begin{bmatrix}
  -1 \\
  1 \\
  1
\end{bmatrix}, \quad
 E = \begin{bmatrix}
  1 \\
   0 \\
   1
   \end{bmatrix}, \\ 
C &= \begin{bmatrix}
 1   &  1 &    0 \\
     1   & -1   &  1
     \end{bmatrix}, \quad
D = \begin{bmatrix}
     2  \\
     1   
     \end{bmatrix}, \quad
F = \begin{bmatrix}
 1 \\
 1  
     \end{bmatrix}.
\end{align*}
Historical inputs data are randomly generated, uniformly in the interval $(-4,4)$ for the known input $u(t)$, and in the interval $(-3,3)$ for the disturbance $d(t)$. The time interval of the offline experiment is $T=11$. We collect the data corresponding to the state/input/output trajectories generated by the system driven by the previous inputs and with random initial condition. We group them into the matrix $\Phi_d$. Then, we determine a matrix $\Psi \triangleq \begin{bmatrix} V_p & V_f & W_p & W_f & R_p & R_f \end{bmatrix}$ such that \eqref{ker_im} holds. A possible choice is: 
\begin{eqnarray*}
\Psi \!\!\!\!&=&\!\!\!\! \begin{bmatrix} V_p & V_f & W_p & W_f & R_p & R_f \end{bmatrix} \\
 \!\!\!\!&=&\!\!\!\! \left[\begin{array}{ccc|ccc|c|c|cc|cc}
4  &  3  & 2 &  -1 &  -2  & 1 &  0 & 0 &  0 & 0 &  0 & 0 \\
2  &  1 & 1  &  0 &  -1 & 0  &  1  & 0 &  0 & 0 &  0 & 0 \\
6   & 3  &  2  & -1  & -3&  0  & 0 &  0 &  1 & 0 &  0 & 0 \\
4  &  4  & 0 & -1  &-2  & 0 & 0  & 0 &  0 & 1 &  0 & 0 \\
4  & 3  &  2  & -1 &  0 &  0 & 0 &   1 &  0 & 0 &  -1 & 1
\end{array}\right].
\end{eqnarray*}
Once we have determined $\Psi$, we look for a matrix $\Omega$ that solves \eqref{sol_cond} with $A^*$ Schur stable. To this end, we follow the procedure proposed in Remark \ref{practical}. We verify that $V_f$ has full column rank and then we compute $\bar A$ and $\bar C$, as described in the remark, obtaining: 
\begin{eqnarray*}
\bar A &=& \begin{bmatrix}
   -3.2941 &  -2.9412  & -1.2353 \\
   -0.8235  & -0.2353  & -0.0588 \\
   -0.9412 &  -0.4118  &  0.6471
\end{bmatrix}, \\
\bar C &=& \begin{bmatrix}
   -0.9378  &  0.8800  & -1.4951 \\
    1.3943  &  0.7607 &   1.1882
\end{bmatrix}.
\end{eqnarray*}
It is possible to show that the pair $(\bar A,\bar C)$ is observable and hence there exists a matrix $L$ such that $\bar A + L \bar C$ is Schur stable. For instance, if we impose 
$$
L = \begin{bmatrix}
    1.1351 &   2.8592 \\
    0.0810 &   0.4450 \\
    0.3964  &  0.5414 
\end{bmatrix},
$$
we obtain 
$$
A^* = A_{UIO} =  -(\bar A +L \bar C) = 
\begin{bmatrix}
    0.3721 &  -0.2326  & -0.4651 \\
    0.2791 &  -0.1744  & -0.3488 \\
    0.5581 &  -0.3488  &  -0.6977
\end{bmatrix},
$$
whose eigenvalues are $0,0$ and $0.5$. \\
The corresponding $\Omega$ is: 
$$
\Omega = \begin{bmatrix}
   0  & 2.9767 &  -1.1628 &   0.2558 &  -0.0930 \\
  0 &   0.2326  & -0.3721  & -0.0581  &  0.4302 \\
    1   & 0.4651  & -0.7442  & -0.1163 &  -0.1395
\end{bmatrix}.
$$
Finally, by making use of \eqref{multiplo} and \eqref{Si}, we deduce also the other matrices of the obtained UIO:
\begin{align*}
& B_{UIO}^u = \begin{bmatrix}
   -2.9070 \\
   -0.1802 \\
   -0.3605
\end{bmatrix},   B_{UIO}^y = \begin{bmatrix}
    1.0930 &  -0.1860 \\
    0.3198  &  0.1105 \\
    0.6395   & 0.2209
   \end{bmatrix}, \\
&D_{UIO}^u = \begin{bmatrix}
    0.0930 \\
   -0.4302 \\
    0.1395
     \end{bmatrix},  
D_{UIO}^y = \begin{bmatrix}
   -0.0930   & 0.0930 \\
    0.4302  & -0.4302 \\
   -0.1395  &  0.1395
     \end{bmatrix}.
\end{align*}
The dynamics of the state estimation error is illustrated in Figure \ref{fig2}. 
\begin{figure} 
\centering
\includegraphics[width = 0.49\textwidth]{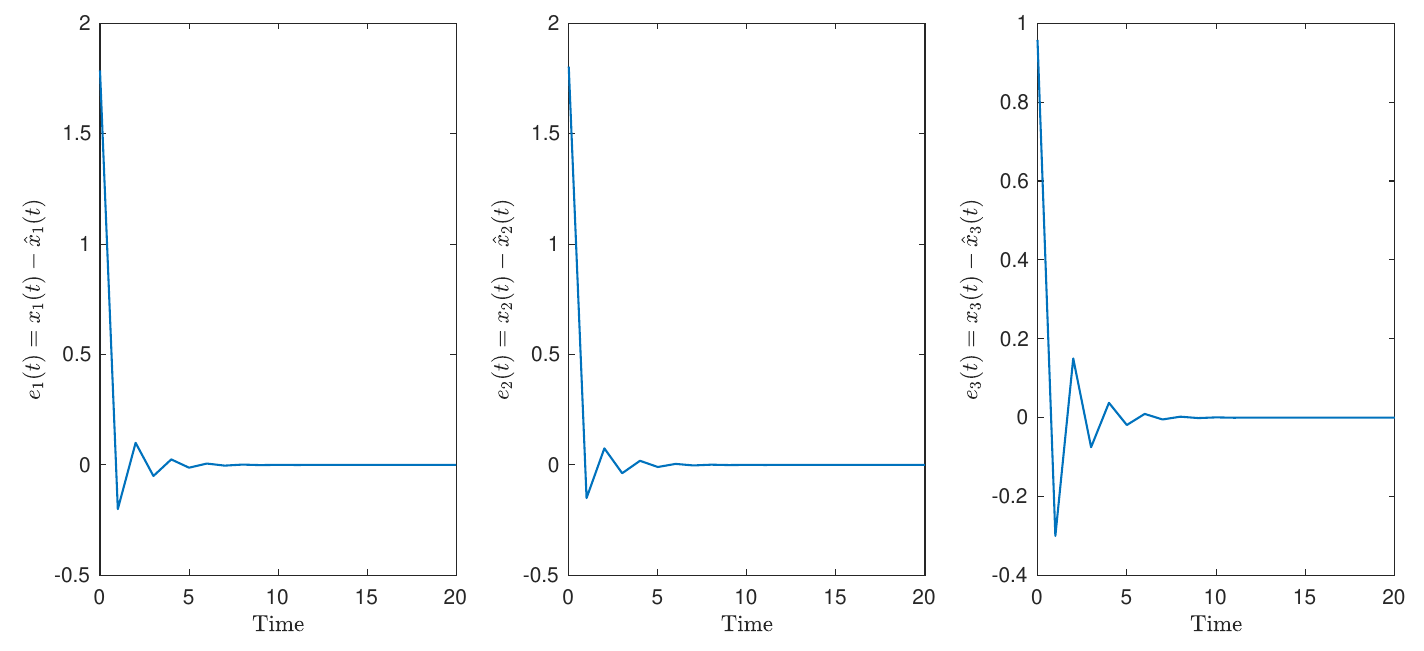}
\caption{Dynamics of the state estimation error} 
\label{fig2}
\end{figure} 
\end{example}
\medskip

\section{Conclusions} \label{Concl} 

In this paper we have proposed a novel strategy, based on Willems' behavioral approach,  to address the design of    unknown-input observers for   linear time-invariant discrete-time state-space models, subject to unknown disturbances. The problem is first addressed assuming that the original system model is known, and later assuming that the model is unknown, but historical data satisfying a certain assumption are available. 
By leveraging fundamental concepts  as the projection of a behavior, the inclusion of a behavior in another one, 
and the use of kernel and image representations, 
we have been able to formalize and solve the UIO design problem.
It is interesting to underline that this seemingly more theoretical approach has the potential to be generalized to more complicated set-ups,
and   provides, as a useful byproduct,  an easy algorithm to design the matrices of a UIO starting from data,
while the solutions proposed in \cite{TAC-UIO,ECC2024} 
are more difficult to implement.
 \medskip

\bibliographystyle{plain}

\bibliography{%alias,Main,FB,BibHK,
Refer177}

\end{document}